\documentclass[12pt,twoside]{amsart}
\usepackage{amssymb}
\usepackage{verbatim}
\usepackage{amsmath}
\usepackage{bm}
\usepackage{a4wide}
\usepackage[latin1]{inputenc}
\usepackage[T1]{fontenc}
\usepackage{times}
\usepackage{amssymb,latexsym}
\usepackage{enumerate}
\usepackage{pict2e}
\usepackage{color}
\usepackage[colorlinks,linkcolor=black,citecolor=black,urlcolor=black]{hyperref}

\makeatletter
\newcommand{\sumprime}{\if@display\sideset{}{'}\sum%
            \else\sum'\fi}
\makeatother

\makeatletter
\DeclareRobustCommand{\intprod}{%
  \mathbin{\mathpalette\int@prod{(0.1,0)(0.9,0)(0.9,0.8)}}%
}
\DeclareRobustCommand{\intprodr}{%
  \mathbin{\mathpalette\int@prod{(0.1,0.8)(0.1,0)(0.9,0)}}}

\newcommand{\int@prod}[2]{%
  \begingroup
  \sbox\z@{$\m@th#1+$}%
  \setlength\unitlength{\wd\z@}%
  \begin{picture}(1,1)
  \roundcap
  \polyline#2
  \end{picture}%
 \endgroup
}
\makeatother

\begin{document}

\numberwithin{equation}{section}

% define theorem environments
\newtheorem{theorem}{Theorem}[section]
\newtheorem{proposition}[theorem]{Proposition}
\newtheorem{conjecture}[theorem]{Conjecture}
\def\theconjecture{\unskip}
\newtheorem{corollary}[theorem]{Corollary}
\newtheorem{lemma}[theorem]{Lemma}
\newtheorem{observation}[theorem]{Observation}
\newtheorem{definition}{Definition}
\numberwithin{definition}{section} %\def\thedefinition{\unskip}
\newtheorem{remark}{Remark}
\def\theremark{\unskip}
\newtheorem{kl}{Key Lemma}
\def\thekl{\unskip}
\newtheorem{question}{Question}
\def\thequestion{\unskip}
\newtheorem{example}{Example}
\def\theexample{\unskip}
\newtheorem{problem}{Problem}

\thanks{Supported by National Natural Science Foundation of China, No. 12271101.}

\address{School of Mathematical Sciences, Key Laboratory of Intelligent Computing and Applications (Ministry of Education), Tongji University, Shanghai 200092, China}

\email{ypxiong@tongji.edu.cn}

\title{Convexity of the Bergman Kernels on Convex Domains}

\author{Yuanpu Xiong}
\date{}

\begin{abstract}
Let $\Omega$ be a convex domain in $\mathbb{C}^n$ and $\varphi$ a convex function on $\Omega$. We prove that $\log{K_{\Omega,\varphi}(z)}$ is a convex function (might be identically $-\infty$) on $\Omega$, where $K_{\Omega,\varphi}$ is the weighted Bergman kernel. When $\varphi\equiv0$, we prove a Brunn-Minkowski type inequality, which further implies that $K_\Omega(z)^{-\frac{1}{2n}}$ is a convex function if $\Omega$ is convex. Some necessary and sufficient conditions for strict convexity are also obtained.


\end{abstract}

\maketitle

\section{Introduction}

Let $\Omega$ be a domain in $\mathbb{C}^n$ and $\varphi$ a plurisubharmonic (psh) function on $\Omega$. Denote by $A^2(\Omega,\varphi)$ the weighted Bergman space of $L^2$ integrable holomorphic functions with respect to the weight $e^{-\varphi}$ on $\Omega$, i.e.,
\[
A^2(\Omega,\varphi):=\left\{f\in\mathcal{O}(\Omega):\|f\|_{\Omega,\varphi}^2=\int_\Omega|f|^2e^{-\varphi}<+\infty\right\}.
\]
The Bergman kernel $K_{\Omega,\varphi}(\zeta,z)$ of $A^2(\Omega,\varphi)$ is a function on $\Omega\times\Omega$, which is holomorphic with respect to $(\zeta,\bar{z})$. The function $K_{\Omega,\varphi}(z):=K_{\Omega,\varphi}(z,z)$ is a real-analytic function on $\Omega$, which is also called the (diagonal) Bergman kernel function. An equivalent characterization of $K_{\Omega,\varphi}(z)$ is the following extremal property:
\begin{equation}\label{eq:extremal}
K_{\Omega,\varphi}(z)=\sup\left\{|f(z)|^2:f\in{A^2(\Omega,\varphi)},\ \|f\|_{\Omega,\varphi}\leq1\right\}.
\end{equation}
In particular, it follows that $\log{K_{\Omega,\varphi}}(z)$ is always a psh function (it can be identically $-\infty$).

The goal of this note is to show the following

\begin{theorem}\label{th:convex}
If\/ $\Omega\subset\mathbb{C}^n$ is a convex domain, $\varphi$ is a convex function on $\Omega$, then $\log{K_{\Omega,\varphi}(z)}$ is a convex function.
\end{theorem}

Here and in what follows, we say that a function $\phi:\Omega\rightarrow[-\infty,+\infty)$ is convex if $\Omega$ is a convex domain and
\begin{equation}\label{eq:convex_define}
\phi((1-t)z_0+tz_1)\leq(1-t)\phi(z_0)+t\phi(z_1)
\end{equation}
for all $z_0,z_1\in\Omega$ and $t\in[0,1]$. Note that a convex function $\phi$ identically equals to $-\infty$ if $\phi(z_0)=-\infty$ for some $z_0\in\Omega$. Otherwise $\phi$ would be a continuous function on $\Omega$. Thus $\log{K_{\Omega,\varphi}(z)}>-\infty$ if $\Omega$ is bounded and $\varphi(z)>-\infty$. Indeed, it suffices to prove Theorem \ref{th:convex} under these additional assumptions, in view of Ramadanov's theorem.

When $\varphi=0$, we write $K_\Omega$ instead of $K_{\Omega,0}$. In the case that $n=1$ and $\Omega\subsetneq\mathbb{C}$ is a convex domain, $K_\Omega$ is closely related to the hyperbolic metric $ds_\Omega=\lambda_\Omega(z)|dz|$ (normalized with Gauss curvature $-1$), via the following identity:
\[
K_\Omega(z)=\frac{1}{4\pi}\lambda_\Omega(z)^2,\ \ \ \forall\,z\in\Omega.
\]
As a consequence, the convexity of $\log{K_\Omega(z)}$ is equivalent to that of $\log\lambda_\Omega(z)$. The latter was proved by Caffarelli-Friedman \cite{CF85} by using PDE techniques, for $u:=\log\lambda_\Omega(z)$ is the solution of the Liouville equation
\[
\Delta{u}=e^u.
\]
Indeed, the convexity of solutions for more general semilinear elliptic PDEs is established in \cite{CF85}.

A complex-analytic proof was given later by Gustafsson \cite{Gustafsson90}. Recall that a conformal mapping $f:\mathbb{D}\rightarrow\Omega$ is called a convex univalent function if $\Omega\subsetneq\mathbb{C}$ is a convex domain. A classical result (cf. Ahlfors \cite{AhlforsBook}) asserts that $f$ is convex univalent if and only if
\begin{equation}\label{eq:convex_univalent}
\mathrm{Re}\,\frac{zf''(z)}{f'(z)}\geq-1.
\end{equation}
Gustafsson showed that the convexity of $\log\lambda_\Omega(z)$ (i.e., the convexity of $\log{K_\Omega(z)}$) is equivalent to \eqref{eq:convex_univalent}, and hence is equivalent to the convexity of $\Omega$.

It is also interesting to point out that the convexity of $\log\lambda_\Omega(z)$ has been essentially known earlier. Minda-Wright \cite{MW82} proved a seemingly stronger result that $-1/\lambda_\Omega(z)$ is a convex function if $\Omega\subsetneq\mathbb{C}$ is convex. Since $-\log(-t)$ is a convex and increasing function, it follows that $\log\lambda_\Omega=-\log(-(-1/\lambda_\Omega))$ is convex on $\Omega$. Their method is a surprisingly simple application of the Schwarz-Pick lemma. Indeed, the convexity of domain $\Omega$ and functions $\log\lambda_\Omega(z)$ and $-1/\lambda_\Omega(z)$ are all equivalent (see Kim-Minda \cite{KM93} for more equivalent conditions).

For higher-dimensional cases, some crucial difficulties might arise if one wants to extend the above methods. For example, the Riemann mapping theorem fails in the case of several complex variables, and it is also not known whether $K_{\Omega,\varphi}(z)$ satisfies certain elliptic PDE. Instead, we may reduce the proof of Theorem \ref{th:convex} to a celebrated theorem of Berndtsson \cite{Berndtsson06} (see also \cite{Berndtsson09}) on subharmonicity properties of Bergman kernels (the case $n=1$ and $\varphi=0$ is due to Maitani-Yamaguchi \cite{MY04}), by using certain characterization of convexity in terms of subharmonicity (see Lemma \ref{lm:convex_psh}).

It is also known that $\log\lambda_\Omega(z)$ (and hence $\log K_\Omega(z)$) is strictly convex, i.e., strict inequality holds in \eqref{eq:convex_define} for any $z_0,z_1\in\Omega$ and $t\in[0,1]$, if $\Omega$ is a convex domain in $\mathbb{C}$ other than a half plane or an infinite strip (cf. \cite{CF85, Gustafsson90}, see also \cite{MW82}). In the higher-dimensional case, we also have the following

\begin{theorem}\label{th:strict_convex}
Let $\Omega$ be a convex domain in $\mathbb{C}^n$. Then the following properties hold:

\begin{itemize}
\item[$(1)$]
If\/ $\Omega$ is bounded and $\varphi$ is any convex function on $\Omega$ with $\varphi\not\equiv-\infty$, then $\log{K_{\Omega,\varphi}(z)}$ is strictly convex.

\item[$(2)$]
$\log{K_\Omega(z)}$ is strictly convex if and only if $\Omega$ does not contain any real line.
\end{itemize}
\end{theorem}

Note that a convex domain $\Omega\subset\mathbb{C}^n=\mathbb{R}^{2n}$ containing a real line can be written as $\Omega'\times\mathbb{R}$ after a rotation, where $\Omega'$ is a convex domain in $\mathbb{R}^{2n-1}$. In particular, $\Omega$ is a half plane or an infinite strip if $n=1$. Thus Theorem \ref{th:strict_convex}/(2) generalizes the corresponding result for $n=1$. Moreover, it would be an interesting question to consider the equivalent conditions for strict convexity in the weighted case for $\Omega\subset\mathbb{C}^n$.

In the light of Minda-Wright \cite{MW82}, it is natural to ask that whether there exists some $\alpha>0$ such that $-K_{\Omega,\varphi}(z)^{-\alpha}$ is a convex function when $\Omega$ is a convex domain in $\mathbb{C}^n$ and $\varphi$ is a convex function on $\Omega$. When $\varphi\not\equiv0$, this is not true even in the case $n=1$. The following example is pointed out by Xu Wang in NTNU.

\begin{example}
Take $\Omega=\mathbb{C}$ and $\varphi(z)=\phi(x):=x^2$, where $z=x+iy\in\mathbb{C}$. By a result of Berndtsson (cf. \cite[\S 7.1]{Berndtsson13}), we have
\[
K_{\mathbb{C},\varphi}(z)=\int_{\mathbb{R}}e^{tx-\phi(t)^2}=\pi^{1/2}e^{x^2}.
\]
But $-K_{\Omega,\varphi}(z)^{-\alpha}=-\pi^{-\alpha/2}e^{-\alpha x^2}$ is not convex for all $\alpha>0$.
\end{example}

If $\varphi=0$, the answer is positive.

\begin{corollary}\label{cor:Bergman_convex_1}
Let $\Omega\subset\mathbb{C}^n$ be a convex domain, then $-K_\Omega(z)^{-\frac{1}{2n}}$ is a convex function.
\end{corollary}

This follows from the following Brunn-Minkowski type inequality.

\begin{corollary}\label{cor:Bergman_convex_2}
Let $\Omega_0$ and $\Omega_1$ be convex domains in $\mathbb{C}^n$. If $z_0\in\Omega_0$ and $z_1\in\Omega_1$, then
\[
K_{\Omega_0+\Omega_1}(z_0+z_1)^{-\frac{1}{2n}}\geq K_{\Omega_0}(z_0)^{-\frac{1}{2n}} + K_{\Omega_1}(z_1)^{-\frac{1}{2n}}.
\]
\end{corollary}

As an application, Corollary \ref{cor:Bergman_convex_1} implies that $-K_\Omega(z)^{-\frac{1}{2n}}$ is a plurisubharmonic function when $\Omega$ is a convex domain. When $K_\Omega(z)>0$ for all $z\in\Omega$, this is equivalent to the inequality
\[
2n\cdot i\partial\bar{\partial}\log K_\Omega(z)\geq i\partial\log K_\Omega(z)\wedge\bar{\partial}\log K_\Omega(z),
\]
i.e.,
\begin{equation}\label{eq:d_bounded}
|\partial\log K_\Omega(z)|^2_{i\partial\bar{\partial}\log K_\Omega(z)}\leq 2n.
\end{equation}
Note that $K_\Omega(z)>0$ if and only if the convex domain $\Omega$ does not contain any complex line (cf. Nikolov-Pflug \cite{NP03}). As a consequence of \eqref{eq:d_bounded}, we obtain an alternative proof of the known fact that if $\Omega\subset\mathbb{C}^n$ is a convex domain which does not contain any complex line, then the Bergman metric on $\Omega$ is $d$-bounded in the sense of Gromov \cite{Gromov91}. This result can be proved by using the estimates of Nikolov-Pflug \cite{NP03}. Zimmer \cite[Proposition 4.12]{Zimmer21} further generalized to the case of $\mathbb{C}$-convex domains, by using the estimates of Nikolov-Pflug-Zwonek \cite{NPZ11}. It is also worth mentioning that for bounded convex domains, the $d$-boundedness of the Bergman metric can also be proved by using Kim-Zhang \cite[Theorem 1.1]{KZ16}. They proved that any bounded convex domain is uniformly squeezing, while Yeung \cite[Theorem 2]{Yeung09} shows that the uniformly squeezing property implies the $d$-boundedness of the Bergman metric.

We also remark that the exponent $-\frac{1}{2n}$ in Corollary \ref{cor:Bergman_convex_1} is optimal. Indeed, if $\Omega=\mathbb{D}^n$ is the unit polydisc and $z=(t,t,\cdots,t)$, where $t\in[0,1]$, then
\[
K_{\mathbb{D}^n}(t)=\frac{1}{\pi^n(1-t^2)^{2n}}.
\]
When $\alpha>1/{2n}$, $t\mapsto -K_{\mathbb{D}^n}(t)^{-\alpha}$ is not convex.

For the unit ball $\mathbb{B}^n$, it is easy to see that $-K_{\mathbb{B}^n}(z)^{-\frac{1}{n+1}}$ is a convex function. Based on this observation, Wang-Liu-Zhang \cite{WLZ24} asked that whether $-K_{f(\mathbb{B}^n)}(z)^{-\frac{1}{n+1}}$ is a convex function when $f$ is a convex univalent function on $\mathbb{B}^n$ (i.e., $f$ maps $\mathbb{B}^n$ biholomorphically to a convex domain).

\section{Convex functions and subharmonic functions}

Let us first consider a smooth function $\varphi$ of one complex variable $t$. The real Hessian of $\varphi$ is the following quadratic form
\[
\eta\mapsto\frac{\partial^2\varphi}{\partial{t}\partial\bar{t}}|\eta|^2+\mathrm{Re}\,\left(\frac{\partial^2\varphi}{\partial{t}^2}\eta^2\right),\ \ \ \eta\in\mathbb{C}.
\]
It is well-known that $\varphi$ is convex if and only if the real Hessian of $\varphi$ is semi-positive. For $\eta\neq0$, set $\lambda=\eta/|\eta|$. Thus $\varphi$ is convex if and only if
\begin{equation}\label{eq:convex_1_dim}
\frac{\partial^2\varphi}{\partial{t}\partial\bar{t}}+\mathrm{Re}\,\left(\lambda^2\frac{\partial^2\varphi}{\partial{t}^2}\right)\geq0
\end{equation}
for all complex number $\lambda$ with $|\lambda|=1$. Since we may always choose $\varphi$ with $\mathrm{Re}\,(\lambda^2\partial^2\varphi/\partial{t}^2)\leq0$, \eqref{eq:convex_1_dim} implies
\[
\frac{\partial^2\varphi}{\partial{t}\partial\bar{t}}\geq0,
\]
i.e., $\varphi$ is a subharmonic function.

Next, let $|\lambda|<1$ and set
\[
s:=\frac{t-\lambda^2\bar{t}}{1-|\lambda|^4},
\]
i.e., $t=t_\lambda(s):=s+\lambda^2\bar{s}$. Based on \eqref{eq:convex_1_dim}, we have the following criterion for convexity.

\begin{lemma}\label{lm:convex_psh}
Let $U$ be a convex domain in $\mathbb{C}$ and $V_\lambda:=\{s\in\mathbb{C}:t=t_\lambda(s)\in{U}\}$. If $\varphi$ is a function on $U$ such that $s\mapsto\varphi(t_\lambda(s))$ is subharmonic on $V_\lambda$ whenever $|\lambda|<1$, then $\varphi$ is convex on $U$.
\end{lemma}

\begin{proof}
First of all, we assume that $\varphi$ is smooth. A straightforward calculation yields
\[
\frac{\partial\varphi(t_\lambda(s))}{\partial{s}}=\frac{\partial\varphi}{\partial{t}}\frac{\partial{t_\lambda(s)}}{\partial{s}}+\frac{\partial\varphi}{\partial\bar{t}}\frac{\partial\overline{t_\lambda(s)}}{\partial{s}}=\frac{\partial\varphi}{\partial{t}}+\bar{\lambda}^2\frac{\partial\varphi}{\partial\bar{t}}
\]
and
\begin{align*}
\frac{\partial^2\varphi(t_\lambda(s))}{\partial{s}\partial\bar{s}}
=&\,\frac{\partial}{\partial\bar{s}}\left(\frac{\partial\varphi}{\partial{t}}\right)+\bar{\lambda}^2\frac{\partial}{\partial\bar{s}}\left(\frac{\partial\varphi}{\partial\bar{t}}\right)\\
=&\,\frac{\partial^2\varphi}{\partial{t}^2}\frac{\partial{t_\lambda(s)}}{\partial\bar{s}}+\frac{\partial^2\varphi}{\partial{t}\partial\bar{t}}\frac{\partial\overline{t_\lambda(s)}}{\partial\bar{s}}+\bar{\lambda}^2\frac{\partial^2\varphi}{\partial{t}\partial\bar{t}}\frac{\partial{t_\lambda(s)}}{\partial\bar{s}}+\bar{\lambda}^2\frac{\partial^2\varphi}{\partial\bar{t}^2}\frac{\partial\overline{t_\lambda(s)}}{\partial\bar{s}}\\
=&\,(1+|\lambda|^4)\frac{\partial^2\varphi}{\partial{t}\partial\bar{t}}+2\mathrm{Re}\,\left(\lambda^2\frac{\partial^2\varphi}{\partial{t}^2}\right).
\end{align*}
Since $s\mapsto\varphi(t_\lambda(s))$ is subharmonic, we have
\[
(1+|\lambda|^4)\frac{\partial^2\varphi}{\partial{t}\partial\bar{t}}+2\mathrm{Re}\,\left(\lambda^2\frac{\partial^2\varphi}{\partial{t}^2}\right)\geq0,\ \ \ \forall\,\lambda:|\lambda|<1
\]
on $U$. Letting $|\lambda|\rightarrow1$, we obtain \eqref{eq:convex_1_dim} for any complex number $\lambda$ with $|\lambda|=1$.

In general, we fix an arbitrary convex subdomain $U'\subset\subset{U}$. By taking $\lambda=0$, we see that $\varphi$ is a subharmonic function on $U$. In particular, it is locally integrable. Take a cut-off function $\kappa\in{C^\infty_0(\mathbb{C})}$ with $\kappa\geq0$ and $\int_{\mathbb{C}}\kappa=1$. Then the convolution
\[
\varphi_j(t)=\int_{\tau\in\mathbb{C}}\varphi(t-j^{-1}\tau)\kappa(\tau)
\]
is a smooth function defined on $U'$ when $j\gg1$, such that $\varphi_j\downarrow\varphi$ as $j\uparrow\infty$. Note that
\[
\varphi_j(t_\lambda(s))=\int_{\tau\in\mathbb{C}}\varphi(t_\lambda(s)-j^{-1}\tau)\kappa(\tau)=(1-|\lambda|^4)\int_{\tau'\in\mathbb{C}}\varphi(t_\lambda(s-j^{-1}\tau'))\kappa\circ{t_\lambda}(\tau'),
\]
where $\tau'=t_\lambda^{-1}(\tau)$ and $1-|\lambda|^4=\det(t_\lambda)$ (the determinant of $t_\lambda$ as a real linear mapping on $\mathbb{R}^2$). It follows that $\varphi_j(t_\lambda(s))$ is subharmonic with respect to $s\in V'_\lambda:=\{s\in\mathbb{C}:t_\lambda(s)\in{U'}\}$. Thus $\varphi_j$ is convex on $U'$, and so is its decreasing limit $\varphi$. Since $U'$ is arbitrary, we complete the proof.
\end{proof}

\section{Berndtsson's theorem for convex domains and weight functions}

To prove Theorem \ref{th:convex}, we shall make use of the following modification of Berndtsson's subharmonicity theorem (cf. \cite{Berndtsson06}).

\begin{theorem}\label{th:Berndtsson_convex}
Let $\widetilde{\Omega}$ be a bounded convex domain in $\mathbb{C}^{n+1}$ and $\widetilde{\varphi}$ a convex function on $\widetilde{\Omega}$. Set
\[
\widetilde{\Omega}_t:=\left\{z\in\mathbb{C}^n:(z,t)\in\widetilde{\Omega}\right\}\ \ \ \text{and}\ \ \ \widetilde{\varphi}_t:=\widetilde{\varphi}(\cdot,t).
\]
Then $t\mapsto\log{K_{\widetilde{\Omega}_t,\widetilde{\varphi}_t}(z)}$ is a convex function on $\{t\in\mathbb{C}:(z,t)\in\widetilde{\Omega}\}$.
\end{theorem}

\begin{proof}
In the case that $\widetilde{\Omega}$ is pseudoconvex and $\widetilde{\varphi}$ is psh, Berndtsson proved that $t\mapsto\log{K_{\widetilde{\Omega}_t,\widetilde{\varphi}_t}(z)}$ is a subharmonic function (cf. \cite{Berndtsson06}, Theorem 1.1). Take any $\lambda\in\mathbb{C}$ with $|\lambda|<1$ and let $t_\lambda$ be given as above. Define a real linear isomorphism $T$ of $\mathbb{C}^{n+1}$ by
\[
T:(z,s)\mapsto(z,t_\lambda(s))=(z,s+\lambda^2\bar{s}).
\]
Set $\widetilde{D}=T^{-1}(\widetilde{\Omega})$, $\widetilde{\psi}=\widetilde{\varphi}\circ{T}$ and
\[
\widetilde{D}_s:=\left\{z\in\mathbb{C}^n:(z,s)\in\widetilde{D}\right\}=\widetilde{\Omega}_{t_\lambda(s)},\ \ \ \widetilde{\psi}_s:=\widetilde{\psi}(\cdot,s)=\widetilde{\varphi}_{t_\lambda(s)}.
\] It follows that $\widetilde{D}$ is still a convex domain in $\mathbb{C}^{n+1}$ and $\widetilde{\psi}$ is convex with respect to $(z,s)$. In particular, $\widetilde{D}$ is pseudoconvex and $\widetilde{\psi}$ is psh with respect to $(z,s)$, so that Berndtsson's theorem applies, i.e.,
\[
s\mapsto\log{K_{\widetilde{D}_s,\widetilde{\psi}_s}(z)}=\log{K_{\widetilde{\Omega}_{t_\lambda(s)},\widetilde{\varphi}_{t_\lambda(s)}}(z)}
\]
is a subharmonic function on $\{s\in\mathbb{C}:(z,s)\in\widetilde{D}\}=\{s\in\mathbb{C}:(z,t_\lambda(s))\in\widetilde{\Omega}\}$. By Lemma \ref{lm:convex_psh}, $t\mapsto\log{K_{\widetilde{\Omega}_t,\widetilde{\varphi}_t}(z)}$ is a convex function on $\{t\in\mathbb{C}:(z,t)\in\widetilde{\Omega}\}$.
\end{proof}

\section{Proof of Theorem \ref{th:convex}}

\begin{proof}[Proof of Theorem \ref{th:convex}]
We first consider the case that $\Omega$ is a bounded convex domain and $\varphi>-\infty$. Given $z_0,z_1\in\Omega$, set
\[
z_t:=(1-t)z_0+tz_1,\ \ \ \forall\,t\in\mathbb{C}
\]
and $U:=\{t\in\mathbb{C}:z_t\in\Omega\}$. Since $U$ is the pre-image of $\Omega$ under the linear mapping $t\mapsto{z_t}$, it follows that $U$ is a convex neighbourhood of $[0,1]$. Consider the following domain
\[
\widetilde{\Omega}:=\{(z,t)\in\mathbb{C}^{n+1}:z+z_t\in\Omega,\ t\in{U}\}
\]
and the function
\[
\widetilde{\varphi}(z,t):=\varphi(z+z_t),\ \ \ (z,t)\in\widetilde{\Omega}.
\]
It is easy to verify that $\widetilde{\Omega}$ is a convex domain and $\widetilde{\varphi}$ is a convex function on $\widetilde{\Omega}$. By Theorem \ref{th:Berndtsson_convex}, we see that the function $t\mapsto\log{K_{\widetilde{\Omega}_t,\widetilde{\varphi}_t}(z)}$ is convex on $\{t\in\mathbb{C}:(z,t)\in\widetilde{\Omega}\}$. Since
\[
\widetilde{\Omega}_t=\Omega-z_t:=\{z-z_t:z\in\Omega\},
\]
we have $0\in\widetilde{\Omega}_t$ for all $t\in{U}$ and
\[
K_{\widetilde{\Omega}_t,\widetilde{\varphi}_t}(0)=K_{\Omega,\varphi}(z_t).
\]
Thus we obtain the convexity of
\[
t\mapsto\log{K_{\Omega,\varphi}(z_t)},\ \ \ t\in{U}.
\]
In particular, for $t\in[0,1]$,
\[
\log{K_{\Omega,\varphi}((1-t)z_0+tz_1)}\leq(1-t)\log{K_{\Omega,\varphi}(z_0)}+t\log{K_{\Omega,\varphi}(z_1)},
\]
so that $\log{K_{\Omega,\varphi}(z)}$ is a convex function on $\Omega$.

In general, one can take sequences $\{\Omega_j\}$ of bounded convex domains and $\{\varphi_j\}$ of convex functions with $\varphi_j>-\infty$, such that $\Omega_j\uparrow\Omega$ and $\varphi_j\downarrow\varphi$ as $j\uparrow\infty$. The previous argument yields the convexity of $z\mapsto\log K_{\Omega_j,\varphi_j}(z)$. By the Ramadanov theorem, we have
\[
K_{\Omega_j,\varphi_j}(z)\downarrow K_{\Omega,\varphi}(z),\ \ \ \forall\,z\in\Omega
\]
as $j\uparrow\infty$. Thus $z\mapsto\log K_{\Omega,\varphi}(z)$ is a convex function on $\Omega$.
\end{proof}

\begin{remark}
The idea of constructing the domain $\widetilde{\Omega}$ in the proof of Theorem \ref{th:convex} goes back to Oka. It is also used by Maitani-Yamaguchi \cite{MY04} and Berndtsson \cite{Berndtsson06} to prove that $(z,t)\mapsto\log{K_{\widetilde{\Omega}_t,\widetilde{\varphi}_t}}(z)$ is jointly psh. In the setting of Theorem \ref{th:Berndtsson_convex}, one can apply the same trick as in \cite{MY04, Berndtsson06} to prove the joint convexity of $(z,t)\mapsto\log{K_{\widetilde{\Omega}_t,\widetilde{\varphi}_t}}(z)$. This is pointed out to the author by Xu Wang.
\end{remark}

\section{Proof of Theorem \ref{th:strict_convex}}

We start with the following property for strictly convex functions.

\begin{lemma}\label{lm:strict_convex_linear}
If $\phi:\Omega\rightarrow(-\infty,+\infty)$ is a convex function but is not strictly convex, then there exists $z_0,z_1\in\Omega$, such that $\phi$ is a linear function on the interval
\[
[z_0,z_1]:=\{(1-t)z_0+tz_1:t\in[0,1]\},
\]
i.e., $\phi((1-t)z_0+tz_1)=At+B$ for some $A,B\in\mathbb{R}$.
\end{lemma}

\begin{proof}
By \eqref{eq:convex_define}, we see that the maximal principle holds on every interval, i.e.,
\begin{equation}\label{eq:convex_MP}
\phi(z)\leq\max\{\phi(z_0),\phi(z_1)\}
\end{equation}
for all $z_0,z_1\in\Omega$ and $z\in[z_0,z_1]$. Suppose that equality in \eqref{eq:convex_define} holds for some $z_0,z_1\in\Omega$ and $z_t:=(1-t)z_0+tz_1$. By subtracting a linear function from $\phi$, we may assume that $\phi(z_0)=\phi(z_1)=\phi(z_t)=0$. Let $0\leq s\leq t$ and $z_s=(1-s)z_0+sz_1$. It follows from \eqref{eq:convex_MP} that $\phi(z_s)\leq0$. On the other hand, we have
\[
z_t=\frac{1-t}{1-s}z_s+\frac{t-s}{1-s}z_1.
\]
Thus
\[
0=\phi(z_t)\leq\frac{1-t}{1-s}\phi(z_s)+\frac{t-s}{1-s}\phi(z_1)=\frac{1-t}{1-s}\phi(z_s),
\]
i.e., $\phi(z_s)\geq0$. This implies $\phi|_{[z_0,z_t]}=0$, and the same argument also yields $\phi|_{[z_t,z_1]}=0$, so that $\phi|_{[z_0,z_1]}=0$.
\end{proof}

By using an idea in \cite{MW82}, we have the following
\begin{lemma}\label{lm:strict_convex}
Let $\Omega$ be a convex domain in $\mathbb{C}^n$ and $\varphi$ a convex function on $\Omega$. If for any real line $\ell\subset\mathbb{C}^n$, we have
\begin{equation}\label{eq:exhaustive}
\lim_{\ell\ni z\rightarrow{a}}K_{\Omega,\varphi}(z)=+\infty
\end{equation}
for some $a\in\partial\Omega\cap\ell$, then $\log{K_{\Omega,\varphi}(z)}$ is strictly convex.
\end{lemma}

\begin{proof}
If $\log{K_{\Omega,\varphi}(z)}$ is not strictly convex, then there exist $z_0,z_1\in\Omega$, such that $\log{K_{\Omega,\varphi}(z)}$ is a linear function on $[z_0,z_1]$, in view of Lemma \ref{lm:strict_convex_linear}. Let $\ell:=\{(1-t)z_0+tz_1:t\in\mathbb{R}\}$. The real-analyticity of $K_{\Omega,\varphi}(z)$ implies that  $\log{K_{\Omega,\varphi}(z)}$ is also linear on $\Omega\cap\ell$. Thus $\lim_{\ell\ni z\rightarrow{a}}K_{\Omega,\varphi}(z)<+\infty$, which contradicts to \eqref{eq:exhaustive}.
\end{proof}

To verify \eqref{eq:exhaustive}, we shall make use of an estimate of $K_\Omega(z)$ due to Nikolov-Pflug \cite{NP03}. Let $\Omega$ be a convex domain in $\mathbb{C}^n$ which does not contain any real line. For any fixed $z\in\Omega$, there exists a point $z_1\in\partial\Omega$ with
\[
d_1(z):=\delta_\Omega(z)=d(z,\partial\Omega)=|z_1-z|.
\]
Let $H_1$ be the complex hyperplane through $z$ which is orthogonal to the vector $z_1-z$. Then $\Omega_1:=H_1\cap\Omega$ can be identified with a bounded convex domain in $\mathbb{C}^{n-1}$, with $a\in\Omega$. We denote by $\partial_{H_1}\Omega_1$ the boundary of $\Omega_1$ in $H_1$. Take $z_2\in\partial_{H_1}\Omega_1$ with
\[
d_2(z):=\delta_{\Omega_1}(z)=d(z,\partial_{H_1}\Omega_1)=|z_2-z|.\]
This process can be repeated to yield a sequence $H_k$ ($k=1,\cdots,n-1$) of $(n-k)$-dimensional planes through the point $z$, a sequence $\Omega_k$ ($k=1,\cdots,n-1$) of convex domains in $\mathbb{C}^{n-k}$, and a sequence of boundary points $z_1,z_2,\cdots,z_n$, with $z_k\in\partial_{H_{k-1}}\Omega_{k-1}$ and
\[
d_k(z):=\delta_{\Omega_{k-1}}(z)=d(z,\partial{H_{k-1}}\Omega_{k-1})=|z_k-z|.
\]
These sequences might not be uniquely determined by $z$. But for a fixed $z$, we may fix a choice of these sequences. Then it is proved in \cite{NP03} that
\begin{equation}\label{eq:NP}
K_\Omega(z)\geq\frac{1}{(4\pi)^nd_1(z)^2\cdots{d_n(z)^2}}.
\end{equation}
In particular, if $\Omega$ is bounded, then
\[
d_j(z)\leq\mathrm{diam}(\Omega),\ \ \ j=2,3,\cdots,n,
\]
where $\mathrm{diam}(\Omega)=\sup_{z,w\in\Omega}|z-w|$ denotes the diameter of $\Omega$. This combined with \eqref{eq:NP} gives
\begin{equation}\label{eq:NP_1}
K_\Omega(z)\geq\frac{1}{(4\pi)^n\mathrm{diam}(\Omega)^{2n-2}}\frac{1}{\delta_\Omega(z)^2}.
\end{equation}
In particular, $K_\Omega(z)$ is an exhaustive function on $\Omega$.

\begin{proof}[Proof of Theorem \ref{th:strict_convex}/(1)]
It suffices to show that $K_{\Omega,\varphi}(z)$ is also an exhaustive function on $\Omega$, in view of Lemma \ref{lm:strict_convex}. To see this, we apply Hahn-Banach theorem to see that $\varphi$ is bounded from below by a linear function (see, e.g., Theorem 2.18 in \cite{ValentineBook}). Since $\Omega$ is bounded, we see that $\varphi(z)$ is bounded from below, say $\varphi(z)\geq-C$ for all $z\in\Omega$. By using \eqref{eq:extremal}, we have $K_{\Omega,\varphi}(z)\geq{e^{-C}K_\Omega(z)}$, which proves the assertion.
\end{proof}

Next, suppose that the convex domain $\Omega$ is not necessarily bounded and does not contain a real line. Given $z\in\Omega$ and $v\in\mathbb{S}^{2n-1}$, where $\mathbb{S}^{2n-1}$ denotes the unit sphere in $\mathbb{C}^n$, define
\[
L(z,v):=\sup\{t>0:z+sv\in\Omega,\ \forall\,s\in(-t,t)\}.
\]
If $\Omega$ does not contain any real line, then $L(z,v)<+\infty$ for any $z\in\Omega$ and $v\in\mathbb{S}^{2n-1}$. Set
\[
L(z):=\sup_{v\in\mathbb{S}^{2n-1}}L(z,v),
\]
i.e., $L(z)$ is half of the length of the longest segment in $\overline{\Omega}$ whose middle point is $z$. We have
\[
d_j(z)\leq L(z),\ \ \ j=2,3,\cdots,n,
\]
so that
\begin{equation}\label{eq:NP_2}
K_\Omega(z)\geq\frac{1}{(4\pi)^nL(z)^{2n-2}}\frac{1}{\delta_\Omega(z)^2},
\end{equation}
in view of \eqref{eq:NP}. We have

\begin{lemma}\label{lm:bounded_L}
Let $\Omega$ be a convex domain in $\mathbb{C}^n$ which does not contain any real line. Then for any $a\in\partial\Omega$, we have
\[
\limsup_{\Omega\ni z\rightarrow{a}}L(z)<+\infty.
\]
\end{lemma}

\begin{proof}
Suppose the contrary that $\limsup_{\Omega\ni z\rightarrow{a}}L(z)=+\infty$. Then there exist sequences $\{z_k\}\subset\Omega$ and $\{v_k\}\subset\mathbb{S}^{2n-1}$, such that $z_k\rightarrow{a}$ and $L(z_k,v_k)\rightarrow+\infty$. By the compactness of $\mathbb{S}^{2n-1}$, we may also assume that $\{v_k\}$ converges to some $v\in\mathbb{S}^{2n-1}$. Denote $t_k:=L(z_k,v_k)/2$. Since
\[
z_k+sv_k\in\Omega,\ \ \ \forall\,s\in\left(-t_k,t_k\right),
\]
it follows that $a+sv\in\overline{\Omega}$ for all $s\in\mathbb{R}$ by letting $k\rightarrow+\infty$. That is, $\overline{\Omega}$ contains a real line.

Take any $z\in\Omega$. By Theorem 1.11 in \cite{ValentineBook}, we have $(z+w)/2\in\Omega$ for any $w\in\partial\Omega$. Thus
\[
\frac{z+(a+sv)}{2}=\frac{z+a}{2}+\frac{s}{2}v\in\Omega,\ \ \ \forall\,s\in\mathbb{R},
\]
so that $\Omega$ also contains a real line, which is a contradiction.
\end{proof}

\begin{proof}[Proof of Theorem \ref{th:strict_convex}/(2)]
If $\Omega$ is a convex domain in $\mathbb{C}^n$ which does not contain a real line, then \eqref{eq:exhaustive} follows immediately from \eqref{eq:NP_2} and Lemma \ref{lm:bounded_L}. We see that $\log{K_\Omega(z)}$ is strictly convex by Lemma \ref{lm:strict_convex}.

On the other hand, if the convex domain $\Omega$ contains a real line, say $\ell:=\{z_0+tv:t\in\mathbb{R}\}$ for some $v\in\mathbb{S}^{2n-1}$, then we infer from the convexity of $\Omega$ that $z+tv\in\Omega$ for any $z\in\Omega$ and $t\in\mathbb{R}$. Thus for any fixed $t\in\mathbb{R}$, $z\mapsto{z+tv}$ is a biholomorphic automorphism of $\Omega$, so that $K_\Omega(z)=K_\Omega(z+tv)$. It follows that $\log{K_\Omega}(z)$ is a constant on the real line $\ell$, and hence it is not strictly convex.
\end{proof}

\section{Proof of Corollary \ref{cor:Bergman_convex_1} and \ref{cor:Bergman_convex_2}}
\begin{proof}[Proof of Corollary \ref{cor:Bergman_convex_2}]
As in the proof of Theorem \ref{th:convex}, it suffices to consider the case that $\Omega$ is bounded. We shall make use of a standard scaling trick. The Bergman kernel satisfies
\[
K_{c\Omega}(0)=c^{-2n}K_\Omega(0)
\]
whenever $0\in\Omega$ and $c>0$. After translations, we may assume that $z_0=z_1=0$. Set
\[
\widetilde{\Omega}:=\{(z,t);\ 0<\mathrm{Re}\,t<1,\ |\mathrm{Im}\,t|<1,\ z\in(1-\mathrm{Re}\,t)\Omega_0+(\mathrm{Re}\,t)\,\Omega_1\}.
\]
Then $\widetilde{\Omega}$ is a bounded convex domain in $\mathbb{C}^{n+1}$. By Theorem \ref{th:Berndtsson_convex}, for any $t\in(0,1)$, we have
\[
K_{\Omega_t}(0)=K_{(1-t)\Omega_0+t\Omega_1}\leq\max\{K_{\Omega_0}(0),K_{\Omega_1}(0)\}.
\]
Denote $c_0:=K_{\Omega_0}(0)^{-\frac{1}{2n}}$ and $c_1:=K_{\Omega_1}(0)^{-\frac{1}{2n}}$. If we replace $\Omega_0$ by $c_0^{-1}\Omega_0$ and $\Omega_1$ by $c_1^{-1}\Omega_1$, then
\[
K_{(1-t)c_0^{-1}\Omega_0+tc_1^{-1}\Omega_1}(0)\leq \max\{K_{c_0^{-1}\Omega_0}(0),K_{c_1^{-1}\Omega_1}(0)\}=1.
\]
Take $t=c_1/(c_0+c_1)$. It follows that
\[
1\geq K_{\frac{\Omega_0+\Omega_1}{c_0+c_1}}(0)=(c_0+c_1)^{2n}K_{\Omega_0+\Omega_1}(0),
\]
i.e., $K_{\Omega_0+\Omega_1}(0)^{-\frac{1}{2n}}\geq K_{\Omega_0}(0)^{-\frac{1}{2n}} + K_{\Omega_1}(0)^{-\frac{1}{2n}}$.
\end{proof}

\begin{proof}[Proof of Corollary \ref{cor:Bergman_convex_1}]
Fix $z_0,z_1\in\Omega$, we have
\[
K_{\Omega+\Omega}(z_0+z_1)^{-\frac{1}{2n}}\geq K_\Omega(z_0)^{-\frac{1}{2n}} + K_\Omega(z_1)^{-\frac{1}{2n}}.
\]
Since $\Omega$ is a convex domain, it is easy to verify that $\Omega+\Omega=2\Omega$, so that
\[
K_\Omega\left(\frac{z_0+z_1}{2}\right)^{-\frac{1}{2n}}=\left(2^{2n}K_{\Omega+\Omega}(z_0+z_1)\right)^{-\frac{1}{2n}}\geq \frac{1}{2}K_\Omega(z_0)^{-\frac{1}{2n}}+\frac{1}{2}K_\Omega(z_1)^{-\frac{1}{2n}}.\qedhere
\]
\end{proof}

\subsection*{Acknowledgement}
The author would like to thank Bo-Yong Chen for introducing this topic to him together with constant encouragement. He is also grateful to Xu Wang and Jianfei Wang for valuable comments and useful remarks.

\end{document}